\documentclass[12pt]{amsart}
\usepackage{amscd,amssymb,amsthm,amsmath,amssymb,textcomp,supertabular,longtable,enumerate,rotating,mathrsfs,mathtools,hyperref,latexsym,amscd,amsbsy,mathrsfs}
\usepackage{pdflscape}
\usepackage[matrix,arrow,curve]{xy}
\newtheorem{theorem}[equation]{Theorem}

\newtheorem{proposition}[equation]{Proposition}
\newtheorem{lemma}[equation]{Lemma}
\newtheorem{corollary}[equation]{Corollary}
\newtheorem{conjecture}[equation]{Conjecture}

\newtheorem{maintheorem}{Main Theorem}

\theoremstyle{definition}
\newtheorem{example}[equation]{Example}
\newtheorem{definition}[equation]{Definition}

\theoremstyle{remark}
\newtheorem{remark}[equation]{Remark}

\makeatletter\@addtoreset{equation}{section} \makeatother

\swapnumbers

\newtheoremstyle{dotless}{}{}{\rm}{}{\sc}{}{ }{}

\theoremstyle{dotless}
\newtheorem{case}[subsection]{Case:}

\author{Arman Sarikyan}

\title{On the Rationality of Fano-Enriques Threefolds}

\address{\emph{Arman Sarikyan}
\newline
\textnormal{School of Mathematics, The University of Edinburgh,  Edinburgh EH9 3JZ, UK.}
\newline
\textnormal{\texttt{a.sarikyan@ed.ac.uk}}}

\renewcommand{\P}{\mathbb{P}}
\renewcommand{\O}{\mathcal{O}}
\newcommand{\Q}{\mathbb{Q}}
\newcommand{\s}{\mathfrak{s}}
\renewcommand{\H}{\mathcal{H}}
\newcommand{\N}{\mathcal{N}}
\newcommand{\M}{\mathcal{M}}
\renewcommand{\L}{\mathcal{L}}
\renewcommand{\l}{\ell}
\DeclareMathOperator{\mult}{mult} 
\DeclareMathOperator{\Aut}{Aut} 
\DeclareMathOperator{\Prym}{Prym}
\DeclareMathOperator{\nef}{nef}

\begin{document}

\begin{abstract}
A Fano-Enriques threefold is a three-dimensional non-Gorenstein Fano variety of index $1$ with at most canonical singularities. We study the birational geometry of Fano-Enriques threefolds with terminal cyclic quotient singularities. We investigate their rationality, and also provide an example of a Fano-Enriques threefold, whose pliability is $9$, i.e. a Fano-Enriques threefold birationally equivalent to exactly $9$ Mori fibre spaces in Sarkisov category.
\end{abstract}

\maketitle

\section{Introduction}
Throughout this paper we work over the field of complex numbers $\mathbb{C}$, and all varieties are assumed to be projective unless stated otherwise. 

Three-dimensional varieties whose hyperplane sections are Enriques surfaces have been first introduced and studied by G. Fano in \cite{fano1938sulle}. He has attempted to classify such varieties, however the provided proof contained inaccuracies due to the lack of a proper theory of three-dimensional birational geometry at that time. Later, the ideas of G.~Fano were summarised and presented in a modern mathematical language by A. Conte and J.~P.~Murre in \cite{conte1985algebraic}. The varieties described above are said to be Fano-Enriques threefolds and their modern definition is the following.

\begin{definition}\label{defin}
	A three-dimensional variety $X$ is called a {\itshape Fano-Enriques threefold} if it has at most canonical singularities, $-K_X$ is not a Cartier divisor and $-K_X \sim_{\Q} H$ for some ample Cartier divisor $H$. The numbers $-K_X^3$ and $g(X)=-\frac{1}{2}K_X^3+1$ are called the {\itshape degree} and {\itshape genus} of $X$ respectively.
\end{definition}

\begin{remark}
	Actually one has $-2K_X\sim 2H$ by the theorems below.
\end{remark}

The connection between the threefolds studied by G. Fano in \cite{fano1938sulle} and Fano-Enriques threefolds is reflected in the following theorems.

\begin{theorem}[{\cite[Proposition 3.3]{prokhorov1995algebraic}}]
	Let $X$ be a Fano-Enriques threefold and $-K_X \sim_{\Q} H$, where $H$ is an ample Cartier divisor. Then a general surface in the linear system $|H|$ is an Enriques surface with at most canonical singularities. It is smooth if the singularities of $X$ are isolated and $-K_X^3\neq 2$.
\end{theorem}

\begin{theorem}[\cite{Cheltsov_1996}]
	Let $X$ be a normal threefold and $H\subset X$ be an Enriques surface
with at most canonical singularities. Assume that $H$ is an ample Cartier divisor on $X$. Then $-2K_X\sim 2H$ and $X$ is either a Fano-Enriques threefold or a contraction of a section of $\mathrm{Proj}\left(\O_H\oplus\O_H\left(H|_H\right)\right)$.
\end{theorem}

Since $-2K_X\sim 2H$ for all Fano-Enriques threefolds, then one can consider the so-called {\itshape canonical covering} of $X$, i.e. the variety $V=\mathrm{Spec}\left(\O_X\oplus\O_X \left(K_X+H\right)\right)$. It can be proved that $V$ is a Fano variety with Gorenstein canonical singularities and the natural double covering $\pi\colon V \to X$ is ramified at finitely many points, which are exactly the non-Gorenstein points of $X$.

Although the classification of Fano-Enriques threefolds is as yet unknown, there are some partial results. For example, it has been proved by I. Cheltsov in \cite{chel1999bounded} that the genus $g(X)$ is bounded by $47$. However, most likely that this bound is far from being sharp, because if $H$ is a smooth Enriques surface, then $g(X)\leq 17$ and this bound is sharp by \cite[Theorem 1.1]{prokhorov2007fano} or \cite[Theorem 1.5]{knutsen2011extendability}. Also, L. Bayle in \cite{bayle} has classified Fano-Enriques threefolds, whose canonical covering $V$ is smooth (see also \cite{sano1995classification}). He has proved the following theorem. 

\begin{theorem}[{\cite{bayle}}]\label{List}
Let $X$ be a Fano-Enriques threefold and let $V$ be its canonical covering. Assume that $V$ is smooth. Then $X$ has eight singular points that are quotient singularities of type $\frac{1}{2}(1,1,1)$, and $V$ is one of the following:
\begin{enumerate}[\normalfont 1)]
	\item the double covering of a quadric ramified in a divisor of degree $8$, $g(X)=2$;
	\item the complete intersection of three quadrics in $\P^6$, $g(X)=3$;
	\item the blow-up of a smooth hypersurface of degree $4$ in $\P(1^4,2)$ along an elliptic curve cut out by two hypersurfaces of degree one, $g(X)=3$;
	\item $\P^1\times S_2$, $g(X)=4$;
	\item the double covering of $\P^1\times \P^1 \times \P^1$ ramified in a divisor of degree $(2,2,2)$, $g(X)=4$;
	\item the blow-up of a smooth complete intersection of two quadrics in $\P^5$ along an elliptic curve cut out by two hyperplane sections, $g(X)=5$;
	\item the hypersurface of degree $4$ in $\P(1^4,2)$, $g(X)=5$;
	\item the complete intersection of three divisors of degree $(1,1)$ in $\P^3\times \P^3$, $g(X)=6$;
	\item $\P^1\times S_4$, $g(X)=7$;
	\item the divisor of degree $(1,1,1,1)$ in $(\P^1)^4$, $g(X)=7$;
	\item the blow-up of the cone over a smooth quadric surface $\P^1\times \P^1\subset \P^3$ along the disjoint union of the vertex and a smooth elliptic curve on $\P^1\times \P^1$, $g(X)=8$;
	\item the compete intersection of two quadrics in $\P^5$, $g(X)=9$;
	\item $\P^1\times S_6$, $g(X)=10$;
	\item $\P^1\times \P^1 \times \P^1$, $g(X)=13$,
\end{enumerate}
where $S_d$ is a smooth del Pezzo surface of degree $d$. 

Moreover, in each deformation family listed above, there exist a smooth Fano threefold~$V$ and an involution $\sigma\in \Aut(V)$ that fixes finitely many (actually eight) points on $V$, such that the quotient $V/\langle \sigma\rangle$ is a Fano-Enriques threefold.
\end{theorem}

\begin{remark}
	By \cite{zbMATH01382698} all Fano-Enriques threefolds with at most terminal singularities admit a $\Q$-smoothing, i.e. any such variety is a deformation of one from Theorem \ref{List}.
\end{remark}

 The rationality of Fano-Enriques threefolds is also an open problem that goes back to the works of G. Fano and F. Enriques. The following classical result about it is due to L.~P. Botta and A. Verra.
 
\begin{theorem}[{\cite{botta1983non}}]\label{botta}
Let $X$ be a general Fano-Enriques threefold, whose canonical covering belongs to the deformation family $5)$ from the Theorem \ref{List}. Then $X$ is not rational.
\end{theorem}

\begin{remark} More precisely, L. P. Botta and A. Verra have considered a threefold $Y\subset \P^4$ defined by
	\[
	x_1x_2x_3x_4\left( x_0^2 + x_0\sum_{i=1}^{4}a_ix_i +\sum_{i,j=1}^4b_{ij}x_ix_j\right) + c_1x_2^2x_3^2x_4^2 + c_2x_1^2x_3^2x_4^2 + c_3x_1^2x_2^2x_4^2 +c_4x_1^2x_2^2x_3^2=0,
	\]
	where $a_i$, $b_{ij}$, $c_{i}$ are sufficiently general complex numbers, and proved that $Y$ is not rational. However, one can see that the normalisation of $Y$ is the Fano-Enriques threefold corresponding to the deformation family $5)$ from the Bayle's list.
\end{remark}

\begin{theorem}[{\cite{cheltsov1997rationality}}]
	Let $X$ be a Fano-Enriques threefold. If $g(X)\geq 6$, then $X$ is rational.
\end{theorem}

In the case when $V$ is smooth, more can be said regarding the rationality of $X$.

\begin{theorem}[{\cite[Corollary 18]{cheltsov2004rationality}}]
	Let $X$ be a Fano-Enriques threefold. Assume that $X$ is not rational and its canonical covering $V$ is smooth. Then $V$ belongs to the deformation families  $1)$, $2)$ or $5)$ from Theorem \ref{List}.
\end{theorem}

In the same paper I. Cheltsov stated the following conjecture.

\begin{conjecture}[{\cite[Conjecture 19]{cheltsov2004rationality}}]
	Let $X$ be a Fano-Enriques threefold such that its canonical covering $V$ is smooth. Let $V$ be one of the following:
	\begin{itemize}
		\item[\textup{i)}] the double covering of a quadric ramified in a divisor of degree $8$;
		\item[\textup{ii)}] the complete intersection of three quadrics in $\P^6$;
		\item[\textup{iii)}] the double covering of $\P^1\times \P^1 \times \P^1$ ramified in a divisor of degree $(2,2,2)$.
	\end{itemize}
	Then $X$ is not rational.
\end{conjecture}

We know from Theorem \ref{botta} that this conjecture holds for a general $X$ if its canonical covering belongs to the deformation family 5) of Theorem \ref{List}. The goal of this paper is to study the rationality of the remaining deformation families and give the positive answer to the above conjecture for a general member. The main results obtained in this paper are the following.
\begin{maintheorem}\label{main1} 
	Let $X$ be a general Fano-Enriques threefold, whose canonical covering $V$ is the smooth intersection of three quadrics in $\P^6$. Then $X$ is not rational.
\end{maintheorem}

Recall from \cite{corti2004birational} that the {\itshape pliability} of a variety $X$ is defined to be the set  of all Mori fibre spaces birational to $X$ up to a square equivalence (see Definition \ref{definition:square}) in Sarkisov category. It is denoted as $\mathcal{P}(X)$.

\begin{maintheorem}\label{main2} 
	Let $X$ be a Fano-Enriques threefold, whose canonical covering $V$ is the double covering of a quadric ramified in a divisor of degree $8$. Then there exist eight Sarkisov links of type I:
		\[
			\xymatrix{
				&\widetilde{X}_i\ar@{-->}^{\psi_i}[r] \ar_{\kappa_i}[ld]&U_i\ar^{f_i}[d]\\
				X&&\P^1
			}	
		\]
		where $\kappa_i$ is the Kawamata blow-up of a non-Gorenstein point $q_i$, the map $\psi_i$ is a flop and $f_i \colon U_i\to \P^1$ is a Mori fibre space, whose general fibre is a del Pezzo surfaces of degree $1$ for $i =1,\ldots,8$.
		
	Moreover, let $\Phi\colon X \dashrightarrow W$ be a birational map to a Mori fibre space $W/B$. Then one has the following possibilities:
		   \begin{itemize}
				\item $B=\mathrm{pt}$ and $\Phi$ is an isomorphism;
				\item $B\cong \P^1$ and for $i= 1,\ldots,8$ there exists the following commutative diagram:
				\[
					\xymatrix{X\ar@{-->}_{\rho_i}[dr] \ar@{-->}@/^1pc/[rrd]^{\Phi}&\\
					&U_i\ar@{-->}[r]^{\chi} \ar[d]_{f_i}&W\ar[d]\\
					&\P^1\ar^{\omega}[r]&B
					}
   				 \]	
    		where $\chi\colon U_i \dashrightarrow W$ is a birational map, $\rho_i=\psi_i \circ \kappa_i^{-1}$ and $\omega \colon \P^1\to B$ is an isomorphism.
			\end{itemize}
\end{maintheorem}

As a direct corollary from Main Theorem \ref{main2} we have the following.

\begin{corollary}
	In the notation and assumptions of Main Theorem \ref{main2} the following holds: 
	\begin{itemize}
		\item[\textup{1)}] $2\leq|\mathcal{P}(X)|\leq 9$. For a general $X$ one has $|\mathcal{P}(X)|=9$;
		\item[\textup{2)}] $\Aut(X)=\mathrm{Bir}(X)$;
		\item[\textup{3)}] $X$ is not birational to a conic bundle;
		\item[\textup{4)}] $X$ is not rational.
	\end{itemize}
\end{corollary}

\begin{corollary}
	Let $X$ be a general Fano-Enriques threefold such that its canonical covering $V$ is smooth. Then $X$ is not rational if and only if $V$ is one of the following:
	\begin{itemize}
		\item[\textup{i)}] the double covering of a quadric ramified in a divisor of degree $8$;
		\item[\textup{ii)}] the complete intersection of three quadrics in $\P^6$;
		\item[\textup{iii)}] the double covering of $\P^1\times \P^1 \times \P^1$ ramified in a divisor of degree $(2,2,2)$.
	\end{itemize}
\end{corollary}

{\bfseries Acknowledgments.} I am grateful to I. Cheltsov for introducing me to this topic and for careful support. I am also grateful to K. Loginov and A. Trepalin for useful conversations.

\section{The Smooth Intersection of Three Quadrics}

\subsection{The Rationality of Conic Bundles}
For many years the rationality of conic bundles has been well studied. In this section we would like to recall some classical results on conic bundles, which will be used later. For more details on the rationality of conic bundles we refer the reader to \cite{prokh_conics}.

Let $f\colon Y\to S$ be a conic bundle, where $\dim Y=3$ and $\dim S=2$. Recall that a conic bundle $f\colon Y\to S$ is said to be a {\itshape standard} conic bundle if both $Y$ and $S$ are smooth and the relative Picard group of $Y$ over $S$ has rank $1$. By \cite[Theorem 1.11]{sarkisov1983conic} for a given conic bundle $f\colon Y\to S$ there exists the following commutative diagram:
\[
\xymatrix{
&Y\ar[d]_{f}\ar@{-->}[r]^{\phi} &Y'\ar[d]^{f'}\\
&S\ar@{-->}[r]^{\psi}&S'
}
\]
such that $\phi$ and $\psi$ are birational maps and $f'\colon Y'\to S'$ is a standard conic bundle. Thus, without loss of generality, we may assume that $f\colon Y\to S$ is already a standard conic bundle.

Let 
\[
\Delta \vcentcolon =\{s\in S \mid f^{-1}(s) \text{ is a degenerate conic}\}.
\]
Then $\Delta$ is a divisor on $S$ and it is called the {\itshape degeneration curve} of $f$. By \cite[Proposition 1.8] {sarkisov1983conic} the singular points of $\Delta$ correspond to the fibres of $f$ that are double lines, and  by \cite[Corollary 1.11]{sarkisov1983conic} the curve $\Delta$ is nodal. Further, let $\widetilde{\Delta}$ be the Hilbert scheme of lines in the fibres of $f$ over $\Delta$. Then  $f$ induces a double covering $\tilde{f}\colon \widetilde{\Delta} \to \Delta$ ramified in the singular points of the curve $\Delta$, and hence one can consider a principally polarised abelian variety $\Prym(\widetilde{\Delta},\Delta)$, the {\itshape Prym variety} of the pair $(\widetilde{\Delta},\Delta)$, see \cite{beauville1977varietes}, \cite{shokurov1984prym}.

\begin{theorem}[{\cite[Proposition 2.8]{beauville1977varietes}}]
	In the above notation, let $\mathrm{J}(Y)$ be the intermediate Jacobian of $Y$. Then $\mathrm{J}(Y)\cong \Prym(\widetilde{\Delta},\Delta)$ as principally polarised abelian varieties.
\end{theorem}

	Recall that a curve $\Delta$ is said to be {\itshape quasi-trigonal} if $\Delta$ is obtained from a hyperelliptic curve by identifying two smooth points.

\begin{theorem}[{\cite{shokurov1984prym}, \cite{beauville1977varietes}}]\label{theorem:prym}
	The Prym variety $\Prym(\widetilde{\Delta},\Delta)$ is a sum of Jacobians of smooth curves if and only if $\Delta$ is either
	\begin{itemize}
		\item hyperelliptic;
		\item trigonal;
		\item quasi-trigonal;
		\item a plane quintic curve and the corresponding double covering
is given by an even theta-characteristic.
	\end{itemize}
\end{theorem}

\subsection{Conic Bundles Birational $X$}\label{invol}

During this section let $X$ be a Fano-Enriques threefold, whose canonical covering $V$ is the smooth intersection of three quadrics in $\P^6$, and denote by $\sigma \colon V \to V$ the involution corresponding to the double covering $\pi\colon V\to X$.

By \cite[Section 6.1.5]{bayle} there exist homogeneous coordinates $x_0,\ldots,x_6$ on $\P^6$ such that $V\subset \P^6$ is defined by the following system of equations:
			\[
				\begin{cases}
					P_1(x_0,x_1,x_2,x_3)+Q_1(x_4,x_5,x_6)=0,\\
					P_2(x_0,x_1,x_2,x_3)+Q_2(x_4,x_5,x_6)=0,\\
					P_3(x_0,x_1,x_2,x_3)+Q_3(x_4,x_5,x_6)=0,
				\end{cases}
			\]
			where $P_i$ and $Q_i$ are homogeneous polynomials of degree $2$, and $\sigma$ is given by 
			\[
				\sigma (x_0:x_1:x_2:x_3:x_4:x_5:x_6)=(x_0:x_1:x_2:x_3:-x_4:-x_5:-x_6).
			\]
The involution $\sigma$ fixes exactly eight points $p_i$ for $i=1,\ldots,8$ on $V$, which are given by
			\[
				V\cap\{x_4=x_5=x_6=0\}.
			\]
The smoothness of $V$ implies that the intersection $V\cap\{x_0=x_1=x_2=x_3=0\}$ is empty, i.e.
\[
\bigcap_{i=1}^3 \{Q_i(x_4,x_5,x_6)=0\}=\varnothing.
\]

Denote by $q_i$ for $i=1,\ldots,8$ the singular points of $X$, i.e. $q_i=\pi(p_i)$. Notice that they are all of the type $\frac{1}{2}(1,1,1)$. Moreover, by changing the coordinates system, we can assume that
\begin{align*}
&p_1=(1:0:0:0:0:0:0);\\
&p_2=(0:1:0:0:0:0:0);\\
&p_3=(0:0:1:0:0:0:0);\\
&p_4=(0:0:0:1:0:0:0),
\end{align*}
and hence $P_n (x_0,x_1,x_2,x_3)$ do not contain monomials of the form $\alpha x_k^2$.

Let's consider the linear subsystem $\L=\langle x_0,x_1,x_2,x_3\rangle|_V \subset |-K_V|$. Notice that one has  $\L=\pi^*|H|$ in the notation of Definition \ref{defin}. Then $\L_{p_1}=\langle x_1,x_2,x_3\rangle|_V \subset \L$, the linear subsystem of sections of $\L$ vanishing at $p_1$, yields a rational map 
\[
	\phi_{\L_{p_1}}\colon V \dashrightarrow \P^2,
\]
undefined at $p_1$ only. Further, denote by $\L_{p_1,p_r}\subset \L_{p_1}$ the linear subsystem of sections of $\L_{p_1}$ that additionally vanish at $p_r$. 

\begin{proposition}\label{prop:fibres}
For a general $V$ the base locus $\mathrm{Bs}\L_{p_1,p_r}$ consists of four smooth conics.  
\end{proposition}
\begin{proof}
Suppose that  $f_1=\sum_{i=1}^3 a_ix_i$ and $f_2=\sum_{i=1}^3 b_ix_i$ are  linearly independent sections of $\L_{p_1,p_r}$. Let $C_r=\mathrm{Bs}\L_{p_1,p_r}$. Notice that $C_r$ is defined by 
\[
	V\cap\left\{f_1=f_2=0\right\}.
\]
Using the equations $f_1=f_2=0$, we can linearly eliminate two variables, i.e. up to a linear change of coordinates, say $x_2=\lambda x_1$ and $x_3=\mu x_1$ for $\lambda$, $\mu \in \mathbb{C}$. Consequently, we can define the curve $C_r$ in $\P^4$ with coordinates $x_0,x_1,x_4,x_5,x_6$ by the following system:
\[
	\begin{cases}
		P_1(x_0, x_1,\lambda x_1, \mu x_1)+Q_1(x_4,x_5,x_6)=0,\\
		P_2(x_0, x_1,\lambda x_1, \mu x_1)+Q_2(x_4,x_5,x_6)=0,\\
		P_3(x_0, x_1,\lambda x_1, \mu x_1)+Q_3(x_4,x_5,x_6)=0.
	\end{cases}
\]
The polynomials $P_n(x_0, x_1,\lambda x_1, \mu x_1)$ have two non-zero solutions, that correspond to the points $p_1$ and $p_r$. Therefore, they are pairwise linearly dependent. Hence, without loss of generality, we can assume that $C_r$ is defined by
\[
	\begin{cases}
		P_1(x_0, x_1,\lambda x_1, \mu x_1)+Q_1(x_4,x_5,x_6)=0,\\
		\alpha Q_1(x_4,x_5,x_6)+Q_2(x_4,x_5,x_6)=0,\\
		\beta Q_1(x_4,x_5,x_6)+Q_3(x_4,x_5,x_6)=0,
	\end{cases}
\]
where $\alpha$, $\beta\in \mathbb{C}$. For a general choice of $V$, the last two equations give four distinct roots for $(x_4,x_5,x_6)$. Plugging these roots into the first equation, we obtain four smooth reduced conics.
\end{proof}

Hence, we write $C_r=C_{r,1}\cup C_{r,2} \cup C_{r,3} \cup C_{r,4}$, where $C_{r,i}$ are smooth conics.

\begin{proposition}\label{proposition:free}
	The blow-up of the point $p_1$ resolves the indeterminacy of $\phi_{\L_{p_1}}$. More precisely, one has the following commutative diagram:
	\[
	\xymatrix{
		&\widehat{V}\ar[ld]_{\tau} \ar[rd]^{\epsilon}&\\
	V\ar@{-->}[rr]^{\phi_{\L_{p_1}}}&&\P^2
	}
\]
where $\tau\colon \widehat{V}\to V$ is the blow-up of $p_1$ and $\epsilon \colon \widehat{V}\to \P^2$ is a fibration, whose general fibre is an elliptic curve.
\end{proposition}
\begin{proof}
	A general fibre of $\epsilon$ is an elliptic curve by construction. Notice, that it in order to prove that $\epsilon$ is a morphism, it is enough to show that the strict transform of $\L_{p_1}$ on $\widehat{V}$ is base point free on the $\tau$-exceptional divisor $E\cong \P^2$. Let's consider the affine patch $\{x_0\neq 0\}$. Let $u_i=x_i/x_0$ be affine coordinates. Then $V\cap \{x_0\neq 0\}$ is given by
	\begin{align}\label{systemofeq}
	\begin{cases}
		\alpha_1 u_1 +\alpha_2 u_2+ \alpha_3 u_3 +R_1(u_1, u_2, u_3) +Q_1(u_4,u_5,u_6)=0,\\
		\beta_1 u_1 +\beta_2 u_2+ \beta_3 u_3 +R_2(u_1, u_2, u_3)+Q_2(u_4,u_5,u_6)=0,\\
		\gamma_1 u_1 + \gamma_2 u_2 + \gamma_3 u_3 +R_3(u_1, u_2, u_3)+Q_3(u_4,u_5,u_6)=0,
	\end{cases}
	\end{align}
where $\alpha_i$, $\beta_i$, $\gamma_i$ are some complex numbers and $R_i(u_1, u_2, u_3)$ are homogeneous polynomials of degree $2$. Since $V$ is smooth, then the matrix 
	\[
		\begin{pmatrix}
		\alpha_1 & \alpha_2 & \alpha_3\\
		\beta_1 & \beta_2 & \beta_3\\
		\gamma_1 & \gamma_2 & \gamma_3\\
		\end{pmatrix}
	\]
is non-degenerate. So, by a linear change of coordinates, we can bring the system \eqref{systemofeq} into the following form:
	\[
	\begin{cases}
		u_1= \widetilde{Q}_1(u_4,u_5,u_6) +\widetilde{R}_1(u_1, u_2, u_3),\\
		u_2= \widetilde{Q}_2(u_4,u_5,u_6) +\widetilde{R}_2(u_1, u_2, u_3),\\
		u_3= \widetilde{Q}_3(u_4,u_5,u_6) +\widetilde{R}_3(u_1, u_2, u_3),
	\end{cases}
	\]
for some homogeneous polynomials $\widetilde{Q}_i(u_4,u_5,u_6)$ and $\widetilde{R}_i(u_1, u_2, u_3)$ of degree $2$, where the polynomials $\widetilde{Q}_i(u_4,u_5,u_6)$ are linear combinations of $Q_i(u_4,u_5,u_6)$. Notice that $u_4$, $u_5$, $u_6$ are local coordinates in a neighbourhood of the point $p_1$. Consequently, considering the Taylor series at the origin of $u_i$ for $i=1,2,3$, we see that 
\[
u_i= \widetilde{Q}_i(u_4,u_5,u_6) + \text{higher order terms}.
\]
Notice, that one has 
\[
\bigcap_{i=1}^3 \{\widetilde{Q}_i(u_4,u_5,u_6)=0\}=\varnothing.
\]
Consequently, the restriction on $E$ of the strict transform of $\L_{p_1}$ on $\widehat{V}$ is a base point free linear system of conics.
\end{proof}

\begin{remark}\label{remark:curves}
	Let $\omega_i \colon \widehat{V}'\to \widehat{V}$ be the blow-up of the point $\tau^{-1}(p_i)$. Then the  curves $(\omega_i)_*^{-1}(C_{i,j})\cap (\omega_i)_*^{-1}(C_{i,k})$ do not intersect for all $i,j,k$, otherwise we get a contradiction with the smoothness of $V$.
\end{remark}

Thus, by Proposition \ref{proposition:free}, the linear subsystem $\H_{q_1}\subset |H|$ of sections of $|H|$ vanishing at $q_1=\pi(p_1)$ determines a rational map 
\[
	\phi_{\H_{q_1}}\colon X\dashrightarrow \P^2,
\] 
which, after the Kawamata blow-up $\theta \colon \widehat{X}\to X$ of the point ${q_1}$, gives a conic bundle $c:\widehat{X}\to \P^2$. Namely, there exists the following commutative diagram:  

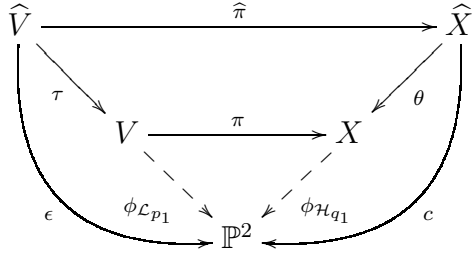
\[
	\xymatrix{
		\widehat{V}\ar@{->}@/_3pc/[rrdd]_{\epsilon}\ar[rrrr]^{\widehat{\pi}}\ar[rd]_\tau&&&&\widehat{X}\ar[ld]^\theta \ar@{->}@/^3pc/[lldd]^{c}\\
		&V\ar[rr]^\pi \ar@{-->}[rd]_{\phi_{\L_{p_1}}}&&X\ar@{-->}[ld]^{\phi_{\H_{q_1}}}&\\
		&&\P^2&&
	}
\]
where $\widehat{\pi}\colon \widehat{V}\to \widehat{X}$ is the induced double covering ramified in the $\tau$-exceptional divisor and in $\tau^{-1}(p_i)$ for $i\neq 1$. However, this conic bundle is non-standard, because $\widehat{X}$ is singular. 
\subsection{From a Non-standard to a Standard Conic Bundle}\label{sect}

In the previous section we arrived at the non-standard conic bundle $c\colon \widehat{X}\to \P^2$. In order to make it standard, firstly we need to resolve the singularities of $\widehat{X}$. More precisely, since there are only eight singular points $q_i$ on $X$ and they are all of the type $\frac{1}{2}(1,1,1)$, then it is enough to consider the Kawamata blow-up of them $\kappa\colon \widetilde{X}\to X$ to obtain a smooth threefold $\widetilde{X}$, see \cite{kawamata1996divisorial}. Notice that the morphism $\kappa$ resolves indeterminacies of $\phi_{\H_{q_1}}$, and therefore  we get a morphism $f\colon \widetilde{X}\to \P^2$, whose general fibre is a conic. However, $f$ is not a standard conic bundle. Indeed, let $s_i=\phi_{\H_{q_1}}(q_i)\in \P^2$, where $i\neq 1$, and let $E_{q_i}$ be the $\kappa$-exceptional divisor over the point  $q_i$. Then we see that $f^{-1}(s_i)$ consists of four disjointed smooth reduced rational curves $Z_i=Z_{i,1} \sqcup Z_{i,2} \sqcup Z_{i,3} \sqcup Z_{i,4}$ by Proposition \ref{prop:fibres} and Remark \ref{remark:curves} and the divisor $E_{q_i}$.
	
 	We now introduce a method that will resolve this problem and provide a standard conic bundle. The method works locally, hence it is enough to treat only one $\kappa$-exceptional divisor $E_{q_i}$.

\begin{lemma} \label{lemma:normal}
	One has $\N_{Z_{i,j}/\widetilde{X}}\cong \O_{\P^1}(-1)\oplus \O_{\P^1}(-1)$.
\end{lemma}
\begin{proof}
	Let $\H_{q_1,q_i}\subset \H_{q_1}$ be the linear subsystem of sections of $\H_{q_1}$ vanishing additionally at $q_i$. Assume $H_{q_1,q_i}$ is a general divisors from $\H_{q_1,q_i}$. Let $S \in |\kappa^*H_{q_1,q_i} - E_{q_1} -E_{q_i}|$ be a general surface. Then $S$ contains the curve  $Z_i$ by construction and $Z_i=S|_S$. Further, suppose that $D\in |H|$ is a surface that does not pass through $q_1$ and $q_i$. Then 
	\[
		 (\kappa^*D - E_{q_1} -E_{q_i})|_S\sim S|_S.
	\]
	Observe that 
	\[
		Z_{i,j} \cdot (\kappa^*D - E_{q_1} -E_{q_i})|_S=1-1-1=-1,
	\]
	for all $j$. Hence,
	\[
	-1 = Z_{i,j} \cdot (\kappa^*D - E_{q_1} -E_{q_i})|_S= Z_{i,j}\cdot Z_i =Z_{i,j}^2.
	\]
	Since $Z_{i,j}\cong\P^1$, then $\N_{Z_{i,j}/S}\cong \O_{\P^1}(-1)$. Notice that locally $Z_{i,j}$ is a complete intersection of two general elements $S_1$, $S_2\in |\kappa^*H_{q_1,q_i} - E_{q_1} -E_{q_i}|$. Consequently, 
	\[
		\N_{Z_{i,j}/\widetilde{X}}\cong \N_{Z_{i,j}/S_1}\oplus\N_{Z_{i,j}/S_2}\cong\O_{\P^1}(-1)\oplus \O_{\P^1}(-1).
	\]
\end{proof}
So, by Lemma \ref{lemma:normal}, flops of the curves $Z_{i,j}$ are the Atiyah flops, i.e. the blow-up of each curve $Z_{i,j}$ and the contraction of the exceptional divisor isomorphic to $\P^1\times \P^1$ along another ruling class. 

Denote by $\psi\colon \widetilde{X}\dashrightarrow U$ the composition of four Atiyah flops of $Z_{i,j}$ and by $g\colon Y\to \P^2$ the blow-up of $s_i$ with the exceptional divisor $B\cong \P^1$. And finally, let $\xi=g^{-1}\circ f \circ \psi^{-1}$ be a rational map $U\dashrightarrow Y$. Therefore, we have the following commutative  diagram:
\[
	\xymatrix{
		&X \ar@{-->}[d]_{\phi_{\H_{q_1}}} && \widetilde{X} \ar[ll]_{\kappa} \ar@{-->}[rr]^{\psi} \ar[dll]_{f} && U \ar@{-->}[dll]^{\xi} \\
		&\P^2&&Y \ar[ll]_{g}&&
	}
\]

\begin{proposition}\label{prop:construction}
	The map $\xi$ is a morphism and $\xi (\psi(E_{q_i}))=B$. Moreover, the restriction  $\xi|_{\psi(E_{q_i})}\colon \psi(E_{q_i})\to B$ is a conic bundle with three singular fibres.
\end{proposition}
\begin{proof}
	Denote by $\omega_1$ the blow-up of $Z_{i,1}$, $Z_{i,2}$, $Z_{i,3}$ and $Z_{i,4}$. Then there is the following commutative diagram for the flop $\psi$:
\[
	\xymatrix{
		&W\ar[ld]_{\omega_1} \ar[rd]^{\omega_2}&\\
	\widetilde{X}\ar@{-->}[rr]^{\psi}&&U
	}
\]
Let $G_{i,j}\cong \P^1\times \P^1$ be the $\omega_1$-exceptional divisor over $Z_{i,j}$ and let $\ell_1^{i,j}$, $\ell_2^{i,j}$ be the ruling classes of $G_{i,j}$ contracted by $\omega_2$, $\omega_1$ respectively. Then
\[
	G_{i,j}|_{G_{i,j}}= -\l_1^{i,j}-\l_2^{i,j}.
\]

Let $\mathcal{D}$ be the pencil of curves on $Y$ that are proper transforms of lines on $\P^2$ passing through the point $s_i$. Note that the class $D+R$, where $D\in \mathcal{D}$ and $R\in|g^*\O_{\P^2}(1)|$ is very ample on $Y$, and the proper transform of $|g^*\O_{\P^2}(1)|$ on $U$ is base point free. Therefore, in order to prove that $\xi\colon U\dashrightarrow Y$ is a morphism, it is enough to show that the proper transform $\mathcal{D}_{U}$ of $\mathcal{D}$ on $U$ is base point free.

Let's first consider the proper transform $\mathcal{D}_W$ of $\mathcal{D}$ on $W$. By construction, the base locus of $\mathcal{D}_W$ is contained in
\[
G_{i,1}\cup G_{i,2}\cup G_{i,3} \cup G_{i,4} \cup E_{q_i}^W,
\]
where $E_{q_i}^W$ is the strict transform of $E_{q_i}$ on $W$. One has 
\[
\mathcal{D}_W\sim (f\circ\omega_1)^*\O_{\P^2}(1)-\sum_{j=1}^4 G_{i,j} -E_{q_i}^W.
\]
Since $E_{q_i}^W|_{G_{i,j}}=\l_2^{i,j}$, then we have $\mathcal{D}_W|_{G_{i,j}}=\l_1^{i,j}$. Consequently, either $\mathrm{Bs}\mathcal{D}_W \cap G_{i,j}$ is empty or all elements of $\mathcal{D}_W$ intersect by the same ruling class $\l_1^{i,j}$. However, the later case is impossible, since the proper transforms of elements of $\H_{q_1}$ on $\widetilde{X}$ intersect transversely at a general point of $Z_{i,j}$.

So, it remains to consider the base locus of $\mathcal{D}_W$ on $E_{q_i}^W$. However, Proposition \ref{proposition:free} implies that the restriction of the linear system   $|\kappa^*H_{q_1} - E_{q_1} -E_{q_i}|$ on $E_{q_i}$ is a linear system of conics on $E_{q_i}$ with four base points that are exactly $E_{q_i}\cap Z_{i,j}$. These points are in general position, because they are determined by the intersection of two conics from the linear system $\langle Q_1(x_4,x_5,x_6),Q_2(x_4,x_5,x_6),Q_3(x_4,x_5,x_6)\rangle$. But we have already proved that there is no base locus of $\mathcal{D}_W$ on $G_{i,j}$. Thereby, $\mathrm{Bs}\mathcal{D}_W=\varnothing$.

Finally, we notice that the restriction of $\mathcal{D}_W$ on $G_{i,j}$ is contained in the fibres of $\omega_2$, consequently $\mathcal{D}_{U}$ is also base point free, so $\xi$ is indeed a morphism.

To prove the last assertion it is enough to see that $\psi(E_{q_i}^W)$ is a del Pezzo surface of degree $5$.
\end{proof}

\begin{remark}
	Notice that a similar construction as in Proposition \ref{prop:construction} can be applied to $V$. So, there is a diagram:
\[
	\xymatrix{
		&&&&W'\ar[dl]_{\omega'_1}\ar[dr]^{\omega'_2}\\
		&V\ar[d]_{\pi}&&\widetilde{V}\ar[ll]_{\widetilde{\kappa}}\ar[d]_{\widetilde{\pi}}&W\ar[dl]_{\omega_1}\ar[dr]^{\omega_2}&U'\ar@{->}@/^2pc/[dd]^{e}\\
		&X \ar@{-->}[drr]_{\phi_{\H_{q_1}}} && \widetilde{X} \ar[ll]_{\kappa} \ar@{-->}[rr]^{\psi} \ar[d]_{f} && U \ar@{->}[d]_{\xi} \\
		&&&\P^2&&Y \ar[ll]_{g}&&
	}
\]
where 
\begin{itemize}
	\item $\widetilde{\kappa}$ is the blow-up of $p_1,\ldots ,p_8$;
	\item $\omega_1'$ is  the blow-up of $\widetilde{\kappa}_*^{-1}(C_{i,1}), \ldots, \widetilde{\kappa}_*^{-1}(C_{i,4})$. Notice, that the $\omega_1$-exceptional divisor over $\widetilde{\kappa}_*^{-1}(C_{i,j})$ is isomorphic to $\P^1\times\P^1$, since on has  
		\[
			\N_{\widetilde{\kappa}_*^{-1}(C_{i,j})}\cong\O_{\P^1}(-2)\oplus \O_{\P^1}(-2)
		\]
		by Lemma \ref{lemma:normal};
	\item $\omega_2'$ is the contraction of the exceptional divisors $\P^1\times \P^1$ along another ruling class to a singular variety $U'$.
	\item $e$ is a morphism, whose fibre is an elliptic curve.
\end{itemize} 
\end{remark}

Thus, applying Proposition \ref{prop:construction} to seven $\kappa$-exceptional divisors $E_{q_i}$ for $i\neq 1$ we obtain a standard conic bundle $\chi \colon \mathfrak{X}\to B$, where $B$ is the blow-up of $\P^2$ in seven points $s_i$. More precisely, one has the following diagram:
\[
	\xymatrix{
		&X \ar@{-->}[d]_{\phi_{\H_{q_1}}} && \widetilde{X} \ar@{->}[ll]_{\kappa} \ar@{-->}[rr]^{\Psi} \ar[dll]_{f} && \mathfrak{X} \ar@{->}[dll]^{\chi} \\
		&\P^2&&B \ar[ll]_{h}&&
	}
\]
where $\Psi$ is a composition of Atiyah flops in curves $Z_{i,j}$ and $h$ is the blow-up of $s_2,\ldots,s_8$.

Let $C$ be the degeneration curve of $f\colon \widetilde{X}\to \P^2$. We want to compute $\deg C$. We write 
\begin{align*}
&\widetilde H_{q_1}\sim_{\Q} \kappa^*H-E_{q_1};\\
&K_{\widetilde{X}}=\kappa^*K_X+\tfrac{1}{2}\sum E_i.
\end{align*}
So,
\[
K_{\widetilde{H}_{q_1}}^2=(K_{\widetilde{X}}+\widetilde{H}_{q_1})^2\cdot\widetilde{H}_{q_1}=-\tfrac{1}{4}E_{q_1}^3=-1.
\]
Hence, by Noether formula $\deg C=9$. 

Now, let's summarise our generality assumptions on $V$. Considering general $V$, we can assume that:
\begin{enumerate}
	\item the points $s_2,\ldots,s_8$ are in general position, i.e. no three of them lie on a line and no six of them lie on a conic;
	\item the base locus $\mathrm{Bs}\L_{p_1,p_i}$ consists of four smooth conics for all $i\neq 1$. Hence, the curve $C$ has ordinary triple points at $s_2,\ldots,s_8$, and the flops $\psi$ are exactly Atiyah flops.
\end{enumerate}

 Consequently, under the generality conditions, one has that $B$ is the del Pezzo surface of degree $2$, the degeneration curve of $\chi$ is $\Delta = h_*^{-1}(C)$ and $\Delta\sim -3K_B$.
 
Now we are ready to prove the main result of this section.

\begin{theorem}\label{non-rational1}
In our generality assumptions $X$ is not rational.	
\end{theorem}
\begin{proof}
	By Theorem \ref{theorem:prym} it is enough to show that the curve $\Delta \subset B$ is neither hyperelliptic, trigonal nor quasi-trigonal.
	
	Notice that $-2K_B$ is very ample. Since $B$ is smooth and $\Delta$ is a Cohen-Macaulay scheme, then we can use the adjunction formula for the dualising sheaf of $\Delta$, that we will denote as $K_{\Delta}$. Therefore, $K_\Delta= (K_B+\Delta)|_{\Delta}=-2K_B|_{\Delta}$, and hence $K_\Delta$ is very ample. Thus, $\Delta$ is not hyperelliptic, otherwise $K_\Delta$ would give a finite morphism of degree $2$.
	
	Let us proof that $\Delta$ is not trigonal. Assume the contrary. Then the image  $\phi_{|K_\Delta|}(\Delta)\subset \P^6$ would have a $3$-section, i.e. a line $\l$ in $\P^6$ such that $|\l \cap \phi_{|K_\Delta|}(\Delta)|=3$. Since one has $H^0(B,\O_B(K_B))=H^1(B,\O_B(K_B))=0$, then the following exact sequence of sheaves
	\[
		\begin{CD}
			0 @>>>\O_B(K_B) @>>> \O_B(-2K_B) @>>> \O_{\Delta}(-2K_B) @>>> 0
		\end{CD}
	\]
	implies that the restriction morphism of global sections 
	\[
		\begin{CD}
		H^0(B,\O_B(-2K_B))	@>>>H^0(\Delta,\O_\Delta(K_{\Delta}))
		\end{CD}
	\]
	is an isomorphism. Consequently, $\phi_{|-2K_B|}(\Delta)=\phi_{|K_\Delta|}(\Delta)$. Now one can show that the image of $\phi_{|-2K_B|}\colon B\hookrightarrow \P^6$ is contained in an intersection of quadrics. Since, as we have noticed, $\phi_{|-2K_B|}(\Delta)=\phi_{|K_\Delta|}(\Delta)$, then $\l\subset \phi_{|-2K_B|}(B)$. However, $\phi_{|-2K_B|}(B)$ obviously does not contain any lines. Contradiction.
	
	Finally, we claim that $\Delta$ is not quasi-trigonal. Indeed, if $\Delta$ was quasi-trigonal, then the image $\phi_{|K_\Delta|}(\Delta)$ would have a $3$-section as well, e.g. \cite[Remark 7.5.1]{prokh_conics}. Contradiction as above.
	\end{proof}

\subsection{An Explicit Example}
Notice that all required generality assumptions on $V$ are Zariski open. Thus, it is enough to provide an example of such a variety $V$, satisfying our assumptions, to finish the proof of the non-rationality of a general $X$ from this family.

Consider $V\subset \P^6$ given by 
\[
	\begin{cases}
		40x_0x_1 + 55x_1x_2 + 12x_1x_3 +3x_2x_3 +x_4^2 + x_5 x_6=0,\\
		12x_0x_2 + 7x_1x_2 + 4x_1x_3 - 9x_2x_3 + x_5^2 + x_4 x_6=0,\\
		5x_1x_2 + 8x_0x_3 + 4x_1x_3 + 9x_2x_3 + x_6^2 +x_4 x_5=0,
	\end{cases}
\]
and the involution $\sigma$ is as in Section \ref{invol}. Then, the $\sigma$-invariant points are the following:
\begin{align*}
		&p_1=(1,0,0,0,0,0,0);\\
		&p_2=(0,1,0,0,0,0,0);\\
		&p_3=(0,0,1,0,0,0,0);\\
		&p_4=(0,0,0,1,0,0,0);\\
		&p_5=(-1,1,1,-1,0,0,0);\\
		&p_6=(2,3,-2,3,0,0,0);\\
		&p_7=(-3,3,1,5,0,0,0);\\
		&p_8=(21/2,9,-10,15,0,0,0).
\end{align*}
We keep the notation from Section \ref{sect}. Let $\pi(p_1)={q_1}$. Now we will explicitly describe the degeneration curve $\Delta$.

By the construction $\L_{p_1}=\langle x_1,x_2,x_3\rangle|_V$. So, in an appropriate affine chart, the section of $\L_{p_1}$ can given by $x_2=\lambda x_1$ and $x_3=\mu x_1$, where $\lambda$, $\mu \in \mathbb{C}$. Thereby, in the affine chart where $x_1\neq 0$, a general fibre of $\phi_{\L_{p_1}}$ is determined by the following equations: 
\[
	\begin{cases}
		40u_0 + 55\lambda + 12\mu  +3\lambda \mu  +u_4^2 + u_5 u_6=0,\\
		12u_0\lambda  + 7\lambda  + 4\mu  - 9\lambda \mu  + u_5^2 + u_4 u_6=0,\\
		5\lambda  + 8u_0\mu  + 4\mu  + 9\lambda \mu  + u_6^2 +u_4 u_5=0,
	\end{cases}
\]
where $(u_0,u_4,u_5,u_6)$ are affine coordinates. Eliminating $u_0$ from the first equation, we have
\[
	\begin{cases}
		7\lambda  - 33/2\lambda ^2 + 4\mu - 63/5\lambda \mu - 9/10 \lambda ^2\mu - 3/10 \lambda  (u_4^2 + u_5 u_6) + u_5^2 + u_4 u_6=0,\\
		5\lambda  + 4 \mu - 2 \lambda  \mu - 12/5 \mu^2 - 3/5 \lambda  \mu^2 - 
 1/5\mu(u_4^2 + u_5 u_6) + u_6^2 +u_4 u_5=0.
	\end{cases}
\]
Let's homogenise the coordinates $u_i$, i.e. let $u_i=y_i/y_1$, where $y_i$ are the homogeneous coordinates. Then
\[
	\begin{cases}
		\left(7\lambda  - 33/2\lambda ^2 + 4\mu - 63/5\lambda \mu - 9/10 \lambda ^2\mu\right)y^2_1 - 3/10 \lambda  (y_4^2 + y_5 y_6) + y_5^2 + y_4 y_6=0,\\
		\left(5\lambda  + 4 \mu - 2 \lambda  \mu - 12/5 \mu^2 - 3/5 \lambda  \mu^2\right)y^2_1 - 
 1/5\mu(y_4^2 + y_5 y_6) + y_6^2 +y_4 y_5=0.
	\end{cases}
\]
Thus, we have defined a general fibre of $\phi_{\L_{p_1}}$ as an elliptic curve embedded into $\P^3$ with coordinates $y_1,y_4,y_5,y_6$ depending on a point $(\lambda,\mu)\in \mathbb{C}^2$. Finally, eliminating $y_1^2$, we find the equation of the quotient of a general fibre of $\phi_{\L_{p_1}}$ by $\sigma$, i.e. the equation of a general fibre of $f\colon \widetilde{X}\to \P^2_{(x_1:x_2:x_3)}$. Homogenising $\lambda$ and $\mu$ back, we have the following equation of a general fibre of $f$:
\begin{align*}\label{fiber}
&\left (3 x_ 2 + 2 x_ 3 \right) \left (5 x_ 1 x_ 2 - 4 x_ 1 x_ 3 + 
   9 x_ 2 x_ 3 \right)\left(y_4^2 + y_5 y_6\right) + \\
 &2 \left (3 x_ 3 -5 x_ 1 \right) \left (5 x_ 1 x_ 2 + 4 x_ 1 x_ 3 + 
   x_ 2 x_ 3 \right)\left(y_5^2 + y_4 y_6\right) +\\
 & \left(-9 x_ 2^2 x_ 3 + 10 x_ 1^2 \left (7 x_ 2 + 4 x_ 3 \right) - 
  3 x_ 1 x_ 2 \left (55 x_ 2 + 42 x_ 3 \right) \right)\left( y_6^2 +y_4 y_5\right)=0.
 \end{align*}
 
	Direct computations from the above equation show that the degeneration curve $C\subset \P^2$ of $f$ is given by:
\begin{align*}
&-109000x_1^6x_2^3 + 1081500x_1^5x_2^4 - 2549250x_1^4x_2^5 + 
 2244375x_1^3x_2^6 - 144000x_1^6x_2^2x_3\\
&+ 2074600x_1^5x_2^3x_3 - 5674500x_1^4x_2^4x_3 + 5284350x_1^3x_2^5x_3 + 358425x_1^2x_2^6x_3 - 48000x_1^6x_2x_3^2 \\
&+ 1278400x_1^5x_2^2x_3^2 -4514680x_1^4x_2^3x_3^2 + 4652580x_1^3x_2^4x_3^2 + 
 312930x_1^2x_2^5x_3^2 + 3645x_1x_2^6x_3^2 \\
&+ 214400x_1^5x_2x_3^3 - 1482560x_1^4x_2^2x_3^3 + 2009592x_1^3x_2^3x_3^3 - 
 322956x_1^2x_2^4x_3^3 - 85374x_1x_2^5x_3^3\\
 &- 9477x_2^6x_3^3 - 25600x_1^5x_3^4 - 140800x_1^4x_2x_3^4 + 551328x_1^3x_2^2x_3^4 - 411984x_1^2x_2^3x_3^4 - 113616x_1x_2^4x_3^4\\
 & - 23328x_2^5x_3^4 + 
 15360x_1^4x_3^5 + 99072x_1^3x_2x_3^5 - 115776x_1^2x_2^2x_3^5 - 
 31104x_1x_2^3x_3^5 - 15552x_2^4x_3^5 \\
 &- 6656x_1^3x_3^6 - 
 6912x_1^2x_2x_3^6 + 2592x_1x_2^2x_3^6 - 3024x_2^3x_3^6=0.
\end{align*}
Indeed, the curve obtained above is irreducible, has the same degree as $C$ and coincides with $C$ on the affine patch by the construction. Hence it coincides with $C$ everywhere.

Now direct computations show that
\begin{itemize}
	\item the points $s_i$ have the following coordinates: $(1,0,0)$, $(0,1,0)$, $(0,0,1)$, $(1,1,-1)$, 
	$(3,-2,3)$, $(3,1,5)$, $(9,-10,15)$, and they are in general position;
	\item the curve $C$ is singular only at the points  $s_i$, and these points are ordinary triple points of $C$.
\end{itemize}
Consequently, after the blow-up $h\colon U \to \P^2$ of the points $s_2,\ldots, s_8$, we get a smooth curve $\Delta=h_*^{-1}(C)\in |-3K_U|$ on the del Pezzo surface $U$ of degree $2$, which is the degeneration curve of a standard conic bundle $\chi\colon \mathfrak{X}\to U$. Thereby, $X$ is not rational by Theorem \ref{non-rational1}. Hence, a general member of this deformation family is also  not rational, which completes the proof of Main Theorem \ref{main1}.

\section{The Double Covering of a Quadric Ramified in a Divisor of Degree $8$}
During this section let $X$ be a Fano-Enriques threefold, whose canonical covering $V$ is the double covering of a quadric ramified in a divisor of degree $8$ and $\sigma \colon V \to V$ be the involution corresponding to the double covering $\pi\colon V\to X$ as before. 

 By {\cite[Section 6.1.6]{bayle}} there are homogeneous coordinates $x_0,\ldots, x_4, x_5$ on~$\P(1^5,2)$, where $x_5$ is the coordinate of weight $2$, such that $V\subset \P(1^5,2)$ is defined by the following system of equations:
			\begin{equation*}
				\begin{cases}
					P_2(x_0,x_1,x_2)+R_2(x_3,x_4)=0,\\
					x_5^2+A_2(x_0,x_1,x_2)x_3^2+B_2(x_0,x_1,x_2)x_3x_4 +C_2(x_0,x_1,x_2)x_4^2
					+F_4(x_0,x_1,x_2)+G_4(x_3,x_4)=0,
				\end{cases}				
			\end{equation*}
where $P_2$, $R_2$, $A_2$, $B_2$, $C_2$, are homogeneous polynomials of degree $2$ and $F_4$, $G_4$ are homogeneous polynomials of degree $4$. The involution $\sigma$ is given by 
			\[
				\sigma (x_0:x_1:x_2:x_3:x_4:x_5)=(x_0:x_1:x_2:-x_3:-x_4:-x_5).
			\]
It fixes exactly eight points $p_i$ for $i=1,\ldots,8$ on $V$, which are given by
			\[
				V\cap\{x_3=x_4=x_5=0\}.
			\]
\begin{remark} \label{remark:on_polynom}
Notice that $R_2(x_3,x_4)$ is reduced. Assume the contrary. Then, without loss of generality, we can suppose that $R_2(x_3,x_4)=\mu (x_3-\lambda x_4)^2$ for some $\mu$, $\lambda\in \mathbb{C}$. Let $\beta\in \mathbb{C}$ be a complex number such that 
	\[
		\beta^2 + G_4(\lambda,1)=0.
	\]
	Then the point $(0:0:0:\lambda:1:\beta)\in V$ is singular. Contradiction with the smoothness of $V$. 
\end{remark}

Thus, changing linearly the coordinate system, we can assume that $R_2(x_3,x_4)=x_3x_4$.

\begin{remark}\label{remark:on_polynom2}
	The polynomial $G_4(x_3,x_4)$ is co-prime with $x_3x_4$. Indeed, assume that $G_4(x_3,x_4)$ is divisible by $x_3$. Then the point $(0:0:0:0:1:0)\in V$ is singular. Contradiction. If $G_4(x_3,x_4)$ is divisible by $x_4$, then we get the same contradiction with the point $(0:0:0:1:0:0)\in V$.
\end{remark}
Denote by $q_i$ for $i=1,\ldots,8$ the singular points of $X$, i.e. $q_i=\pi(p_i)$. They are all of the type $\frac{1}{2}(1,1,1)$.

\subsection{Fibrations on Del Pezzo Surfaces Birational to $X$}\label{dp1}
Consider the linear subsystem  $\L=\pi^*|H|=\langle x_0,x_1,x_2\rangle|_V \subset |-K_V|$. Let's fix a $\sigma$-invariant point $p$. Then the linear subsystem $\L_p\subset \L$ of sections of $\L$ vanishing at $p$ yields a rational map 
\[
	\phi_{\L_{p}}\colon V \dashrightarrow \P^1,
\]
undefined in the singular curve $Z=\mathrm{Bs}\L_p$. 

Changing the coordinate system, we can also assume that $p=(1:0:0:0:0:0)$. Let $\tau \colon \widehat{V}\to V$ be the blow-up of the point $p$ with the exceptional divisor $G$, and $\widehat{\sigma}\colon \widehat{V}\to \widehat{V}$ be the lift of the involution $\sigma$. 

\begin{proposition}
	The restriction of the strict transform of $\tau_*^{-1}\L_p$ on $G$ is a mobile linear system of conics.
\end{proposition}
\begin{proof}
Consider the affine patch $\{x_0\neq 0\}$ with affine coordinates 
\[
	u_1=x_1/x_0,\ldots,\, u_4=x_4/x_0,\, u_5=x_5/x_0^2.
\] 
Then one can see that $u_3$, $u_4$ and $u_5$ are the local coordinates of $V\cap\{x_0\neq 0\}$ at the origin. Hence, we can consider them as homogeneous coordinates on $G\cong \P^2$. Thereby, the restriction of $\tau_*^{-1}\L_p$ on $G$ is generated by two conics:
\begin{align*}
	&\mathcal{Q}_1=\{u_5^2+au_3^2+cu_4^2=0\},\\
	&\mathcal{Q}_2=\{u_3 u_4=0\},
\end{align*}
where $a=A_2(1,0,0)$, $c= C_2(1,0,0)$. Now we see that this linear system of conics is mobile on $G$. Moreover, from the above equations one can see that 
\[
\mathrm{Bs}(\tau_*^{-1}\L_p|_G)=
\begin{cases}
4 \text{ points}, \text{ if both } a \text{ and }c\neq 0; \\
3 \text{ points}, \text{ if only one of } a \text{ or } b\neq 0;\\
2 \text{ points}, \text{ if both } a \text{ and } b= 0.\\
\end{cases}
\]
\end{proof}

\begin{proposition}\label{baselocus}
	The curve $Z$ is a union of two curves $Z_1$ and $Z_2$ of arithmetic genus~$1$, intersecting in the singular point $p$. Moreover, let $\widehat{Z}_i$ be the strict transform of $Z_i$ on $\widehat{V}$. Then $C_i=\widehat{Z}_i/\langle \widehat{\sigma}\rangle$ is a smooth rational curve.
\end{proposition}
\begin{proof}
The curve $Z$ is given in $\P(1^3,2)$  with coordinates $x_0,x_3,x_4,x_5$ by the following system of equations:
	\[
		\begin{cases}
			x_3x_4=0,\\
			x_5^2+ax_0^2x_3^2+bx_0^2x_3x_4 +cx_0^2x_4^2+G_4(x_3,x_4)=0.
		 \end{cases}
	 \]	
Notice, that the first equation, restricted to the plane $\{x_5 = 0\}$, defines a conic, that is singular at the point $p$. So, $Z$ is a double covering of the conic by Remark \ref{remark:on_polynom2}. Hence, $Z$ consists of two components, namely $Z_1=Z\cap\{x_3=0\}$ and $Z_2=Z\cap\{x_4=0\}$. Up to a change of the coordinates, it is enough to study $Z_1$ only.
	
	If $c\neq 0$, then, using the second equation, we see that $Z_1$ is a nodal curve of arithmetic genus $1$ with a node at $p$. So, $\widehat{Z}_1$ is a smooth rational curve, intersecting $G$ in two points. Hence, $C_1$ is a smooth rational curve.
	
	If $c=0$, then $Z_1$ can be given in $\mathbb{A}^2$ with coordinates $x,y$ by 
		\[
			y^2=\alpha x^4,
		\]
		where $\alpha \neq 0$ by Remark \ref{remark:on_polynom2}. The involution $\sigma$ acts on $\mathbb{A}^2$ as 
		\[
			(x,y)\mapsto (-x,-y).
		\]
		So, by taking the quotient we obtain a curve $\widetilde{Z}_1\in \mathbb{A}^3$ with coordinates $u,v,w$, where $u=x^2$, $v=y^2$, $w=xy$ given by the following system:
		\[
			\begin{cases}
				uv=w^2,\\
				v=\alpha u^2.
			 \end{cases}
	 	\]	
	 Hence, we have 
		\[
			\alpha u^3=w^2.
		\]
	So, $\widetilde{Z}_1$ is a cuspidal curve.	 Thus, blowing up the singular point, we see that $C_1$ is a smooth rational curve.
\end{proof}

 Therefore, by Proposition \ref{baselocus} the linear subsystem $\H_q\subset |H|$ of sections of $|H|$ vanishing at $q$, where $q=\pi(p)$, yields a rational map 
 \[	
 	\phi_{\H_q}\colon X \dashrightarrow \P^1,
 \] 
 undefined in a curve that splits into a union of two rational curves $C_1$ and $C_2$ after the Kawamata blow-up $\kappa \colon \widetilde{X}\to X$ of the point  $q$. Let $E$ be the $\kappa$-exceptional divisor over $q$. 

\begin{lemma}\label{lemma:normal2}
	One has $\N_{C_j/\widetilde{X}} \cong \O_{\P^1}(-1)\oplus\O_{\P^1}(-1)$.
\end{lemma}
\begin{proof}
	This is similar to Lemma \ref{lemma:normal}. Let $S\in  |\kappa^*H_q - E|$ be a general surface. We observe as in the proof of Lemma \ref{lemma:normal} that $C_j^2=-1$, where $C_j$ is regarded as a divisor on $S$. Since locally $C_j$ is a complete intersection of two elements of $|\kappa^*H_q - E|$, then 
	\[
		\N_{C_j/\widetilde{X}} \cong \O_{\P^1}(-1)\oplus\O_{\P^1}(-1).
	\]
\end{proof}

Consequently, by Lemma \ref{lemma:normal2}, the flops of the curves $C_1$ and $C_2$ are the Atiyah  flops. Denote by $\psi\colon \widetilde{X}\dashrightarrow U$ the composition of two Atiyah flops. Similarly to the proof of Proposition \ref{prop:construction}, the linear system $\psi_*(\kappa^* H_q - E)$ is base point free, and hence gives a morphism $f\colon U\to \P^1$. Since $\mathrm{Pic}(U)\cong \mathbb{Z}^2$, then $f\colon U\to \P^1$ is a Mori fibre space, i.e. a general fibre $F$ of $f$ is a del Pezzo surfaces.

\begin{proposition}
	Let $F$ be a general fibre of $f\colon U\to \P^1$. Then $F$ is a del Pezzo surface of degree $1$. 
\end{proposition} 
\begin{proof}
Let $S\in  |\kappa^*H_q - E|$ be a general surface. First, let's compute $K_S^2$. We write $K_{\widetilde{X}}= \kappa^*K_X+\frac{1}{2}E$ and hence 
\[
K_S^2=(K_{\widetilde{X}}+S)^2\cdot S=-\tfrac{1}{4}E^3=-1.
\]
Notice that $\psi|_S$ contracts $C_1$ and $C_2$ on $S$ and makes $S$ a fibre of $f$, hence $K_F^2=1$.
\end{proof}

\subsection{The Birational Geometry of Singular Del Pezzo Fibrations of Degree $1$}
In the previous section we constructed the Mori fibre space $f\colon U\to \P^1$ that is birationally equivalent to $X$, whose fibre is a del Pezzo surface of degree $1$. The aim of this section is to study the birational geometry of $U$.

The birational geometry of singular del Pezzo fibrations of degree $1$ has been studied in \cite{okada2020birational}, see also \cite{pukhlikov1998birational, MR1773767, MR1992109, MR2261543} for the smooth case. Let's recall some definitions and results. Throughout this section let $g\colon Y \to \P^1$ be a del Pezzo fibration, where $Y$ is not assumed to be smooth, and let $F$ be the fibre class of $g$. Also, we will use the shorthand $W/B$ to denote a Mori fibre space $W\to B$.

\begin{definition}
	\begin{enumerate}
		\item[]
		\item In the notation above, we define the {\itshape nef threshold} of $Y/\P^1$ as
			\[
				\nef(Y/\P^1)=\inf \{ r \mid -K_Y +rF \text{ is nef}\}.
			\]
		\item For a given number $\delta\in \mathbb{R}$, we say that $Y/\P^1$ satisfies the $K_{\delta}^3$-{\itshape condition} if 
			\[
				(-K_Y)^3 +\nef(Y/\P^1)\leq \delta.
			\]
	\end{enumerate}
\end{definition}

\begin{remark}[{\cite[Remark 3.5]{okada2020birational}}]\label{rem}
	Let $R \subset \overline{\mathrm{NE}}(X)$ be the extremal ray which is not generated by a curve contracted by $g$. Let $\xi \in R$ be a class. Then 
		\[
			(-K_Y+\nef(Y/\P^1)F)\cdot \xi=0,
		\]
		and $F\cdot \xi>0$. Thereby,
		\[
		\nef(Y/\P^1)=-\dfrac{-K_Y\cdot\xi}{F\cdot\xi}.
		\]
\end{remark}

Let us remind a classical definition.

\begin{definition} \label{definition:square}
	A birational map $\Phi \colon W\dashrightarrow W'$ between Mori fibre spaces $W/B$ and $W'/B'$ is a {\itshape square equivalence} if there is a birational map $h\colon B\dashrightarrow B'$ such that the following  diagram
	\[
		\xymatrix{
		&W\ar[d]\ar@{-->}[r]^{\Phi} &W'\ar[d]\\
		&B\ar@{-->}[r]^{h}&B'
		}
	\]
commutes, and the induced birational map between generic fibres of $W/B$ and $W'/B'$ is an isomorphism.
	
\end{definition}
\begin{proposition}[{\cite[Proposition 3.6]{okada2020birational}}] \label{okada} Let $Y/\P^1$ be a fibration on del Pezzo surfaces of degree $1$ and $Y$ has only terminal quotient singular points of type $\frac{1}{2}(1,1,1)$. Suppose that there is a birational map $\Phi \colon Y \dashrightarrow W$ to a Mori fibre space $W/B$ and let a linear system $\mathcal{M}_Y\subset |-n K_Y+aF|$ be the strict transform of a very ample complete linear system on $W$. If $a\geq0$ and $Y/\P^1$ satisfies the $K_{3/2}^3$-condition, then $\Phi$ is a square equivalence.
\end{proposition}

\subsection{Birational Models of $X$}
In this section we will completely describe the birational geometry of $X$ and finish the proof of Main Theorem \ref{main2}.

Recall that in Section \ref{dp1} for each singular point $q_i \in X$ we have    constructed a birational map $\rho_i \colon X\dashrightarrow U_i$, where $f_i\colon U_i\to \P^1$ is a fibration on del Pezzo surfaces of degree $1$ over~$\P^1$. 

Let's fix notation. Let $\Phi\colon X\dashrightarrow W$ be a birational map to a Mori fibre space $W/B$ and let $\M_X$ be the strict transform of a very ample complete linear system on $W$. Denote by~$\lambda\in \Q_{>0}$ the positive rational number such that 
\[
	K_X+\lambda \M_X\sim_{\Q}0.
\]
By the Noether-Fano inequality the pair $(X, \lambda \M_X)$ is not canonical if $\Phi$ is not an isomorphism, see \cite[Theorem 4.2]{corti1995factor}. 

Further, let $\Phi_i=\Phi\circ \rho_i^{-1}$ and $\M_{U_i}=(\rho_i)_* \M_X$. Also let $\mu_i\in \Q_{>0}$ be the positive rational number such that
\[
	K_{U_i}+\mu_i \M_{U_i}\sim_{\Q}a_iF_i
\] 
for some $a_i\in \Q$, where $F_i$ is a class of the fibre of $f_i$.

The next lemma is known as Pukhlikov's inequality.
\begin{lemma}[{\cite{pukhlikov}}] \label{lemma:pukhlikov_4n^2}
	Let $p$ be a smooth point of $X$. Assume that $p$ is a center of non-canonical singularities of the pair $(X, \lambda \M_X)$. Then
	\[
		\mult_{p}(M_1\cdot M_2)>4/\lambda^2
	\]
	for any two general elements $M_1$, $M_2 \in \M_X$.
\end{lemma}

\begin{lemma}\label{lem1}
	Let $p$ be a smooth point of $X$. Then $p$ is not a centre of non-canonical singularities.
\end{lemma}
\begin{proof}
	Notice that 
		\[
			\mathscr{G}\coloneqq \pi^*|-2K_X|=\langle x_0^2, x_1^2, x_2^2, x_3^2, x_4^2, x_0x_1,x_0x_2,x_1x_2, x_3x_4\rangle|_V.
		\]
    Let $o=(a_0:\ldots:a_5)\in V$ be a point such that $\pi(o)=p$. Since $p$ is assumed to be a smooth point of $X$, then $o$ is not $\sigma$-invariant. Denote by $\mathscr{G}_o\subset \mathscr{G}$ the linear subsystem of section of $\mathscr{G}$ vanishing at $o$. Hence, up to a permutation of the coordinates, we have two cases:
    \begin{case}{\sc $a_3\neq 0$.}
    	So, we may assume that $a_3=1$. Then direct computations show that 
    	\begin{align*}
    		\mathscr{G}_o=&\langle x_0^2-a_0^2x_3^2, x_1^2-a_1^2x_3^2, x_2^2-a_2^2x_3^2, x_4^2-a_4^2x_3^2,\\
    		&x_0x_1-a_0a_1x_3^2, x_0x_2-a_0a_2x_3^2,x_1x_2-a_1a_2x_3^2, x_3x_4-a_4x_3^2 \rangle|_V.
    	\end{align*}
    	Therefore,
    	\[
    		\mathrm{Bs}\mathscr{G}_o =\{o,\sigma(o), \nu(o), (\nu\circ\sigma)(o)\},
    	\] 
    	where $\nu \colon V\to V$ is the involution induced by the double covering $g\colon V\to Q$ of a quadric $Q\subset \P^4$ ramified in a divisor of degree $8$.
    \end{case}
    
    \begin{case}{\sc $a_5\neq 0$.}
    	If $a_3\neq 0$ or $a_4\neq 0$, then, by the previous case, $\mathrm{Bs}\mathscr{G}_o$ does not contain any curves. Thus, it remains to consider the case when $o=(a_0:a_1:a_2:0:0:a_5)$. Notice, that one of $a_i$ for $i=0, 1, 2$ must be non-zero. Without loss of generality, we can assume that $a_0\neq 0$, hence $a_0=1$. Then
    	\[
    		\mathscr{G}_o=\langle x_1^2-a_1^2x_0^2, x_2^2-a_2^2x_0^2, x_3^2, x_4^2, x_0x_1-a_1x_0^2, x_0x_2-a_2x_0^2,x_1x_2-a_1a_2x_0^2, x_3x_4\rangle|_V.
    	\]
    Consequently, 
   		 \[
    		\mathrm{Bs}\mathscr{G}_o =\{o,\sigma(o)\}.
   		 \] 
    \end{case}
    Thus, we just have shown  that $\mathrm{Bs}\mathscr{G}_o$ contains no curves. Denote by $\mathscr{L}_p\subset |-2K_X|$ the linear subsystem of section of $|-2K_X|$ vanishing at $p$. Then $\mathrm{Bs}\mathscr{L}_p$ contains no curves as well.
    
    Now assume the contrary, i.e. $p$ is a centre of non-canonical singularities of $(X, \lambda \M_X)$. Let $M_1$, $M_2\in \M_X$ and $S\in \mathscr{L}_p$ be general elements. Then
		\[
			4/\lambda^2=S\cdot M_1\cdot M_2>4/\lambda^2.
		\]
	Contradiction with Lemma \ref{lemma:pukhlikov_4n^2}.
\end{proof}
\begin{remark}
	Alternatively, one can pullback the linear system $\M_X$ on $V$ via $\pi$ and prove similarly that if $P$ is a smooth point of $V$, then the pair $(V,\lambda\pi^*(\M_X))$ is canonical at $P$, see \cite[Theorem 3.8]{Iskovskih} or \cite[Theorem 2.2.1]{cheltsov2005birationally} for a modern proof.
\end{remark}
\begin{lemma}\label{lemma:curves}
	Let $C\subset X$ be an irreducible curve that does not pass through the singular points of $X$. Then $C$ is not a centre of non-canonical singularities.
\end{lemma}
\begin{proof}
	Firstly, recall that by definition, if a curve is a centre of non-canonical singularities, then it is irreducible. Suppose now that an irreducible curve $C$ is a centre of non-canonical singularities of $(X,\lambda \M_X)$. Then one has $\mult_C\M_X>1/\lambda$, e.g. by \cite[Exercise 6.18]{rational}. This implies that $H\cdot C=1$. Notice that the morphism $\pi\colon V \to X$ is \'{e}tale outside of the points $p_i$, hence the pair $(V,\lambda\pi^*(\M_X))$ is non-canonical at the curve $Z$, where $Z$ is the preimage of the curve $C$ via $\pi$. Therefore,  $-K_{V}\cdot Z=2$ by the projection formula. On the other hand, the proofs of \cite[Theorem 3.8, Theorem 3.10]{Iskovskih} show that the only centres of non-canonical singularities on $V$ are curves $\mathscr{C}\subset V$ such that $-K_V\cdot \mathscr{C}=1$, which immediately leads to a contradiction. However, the author has omitted several important technical details. So, for convenience of the reader, we will give an independent proof. In order to finish the proof of the lemma we need to show that: 
		\begin{itemize}
			\item if $Z$ is an irreducible curve, such that $-K_{V}\cdot Z=2$, then $Z$ cannot be a centre of non-canonical singularities;
			\item if $Z=Z_1\cup Z_2$, where $Z_2=\sigma(Z_1)$ and $-K_V\cdot Z_i=1$, then $Z_1$ and $Z_2$ cannot simultaneously be centres of non-canonical singularities.
		\end{itemize} 
	
		Assume that $Z$ is irreducible. In this case the proof of \cite[Theorem 2.3.1]{cheltsov2005birationally} implies that if $Z$ is a centre of non-canonical singularities, then $-K_V\cdot Z=1$. Contradiction.
		
		Assume now that $Z=Z_1\cup Z_2$, where $Z_2=\sigma(Z_1)$ and $-K_V\cdot Z_i=1$. Then $Z_i$ is a smooth rational curve. Notice that $Z_1\cap Z_2=\varnothing$, because the curve $C$ does not pass through the singular points of $X$ by assumption. Recall that we denote as $\nu \colon V\to V$ the involution induced by the double covering $g\colon V\to Q$ of a quadric $Q\subset \P^4$. Notice, that the involutions $\sigma$ and $\nu$ do commute, hence either both $Z_1$ and $Z_2$ are contained in the ramification divisor or neither.
		
		Firstly, we prove that $\nu(Z_i)\neq Z_i$, i.e. the curves $Z_1$ and $Z_2$ are not contained in the ramification divisor of $g$. Assume the contrary. Let $Y\subset Q$ be a hyperplane section containing $g(Z_1)$ and $g(Z_2)$. Then $Y\cong \P^1\times \P^1$ and both curves  $g(Z_1)$ and $g(Z_2)$ lie in the same ruling class of $Y$. Let $L$ be a general curve on $V$ such that $g(L)\subset Y$ is a general curve from another ruling class. Then $-K_V\cdot L=2$ and $L$ intersects $Z_1$ and $Z_2$ each in a single point. Let $S$ be a general surface of $\pi^*(\M_X)$. Then one has $L\not\subset S$, because the linear system $\pi^*(\M_X)$ has no fixed components. So,
			\[
			2/\lambda=S\cdot L\geq \sum_{O\in L\cap Z} \mult_O S \mult_O L \geq \sum_{O\in L\cap Z}  \mult_Z S >2/\lambda.
			\]
			Contradiction. Thus, the curves $Z_i$ are not $\nu$-invariant.

		Let $\beta\colon W\to V$ be the blow-up of $Z_1$ and $Z_2$, and $E_1$, $E_2$ be the $\beta$-exceptional divisors over $Z_1$ and $Z_2$ respectively. Let's proof that the divisor $\beta^*(-2K_V)-E_1-E_2$ is nef. Firstly, denote by $\L\subset |-2K_V|$ the linear subsystem of sections of $|-2K_V|$ containing both $Z_1$ and $Z_2$. We want to show that 
		\[
			\mathrm{Bs}\L\subseteq\{Z_1,\nu(Z_1), Z_2,\nu(Z_2)\}.
		\] 
		Indeed, let's consider the linear subsystem $\mathcal{Q}\subset |\O_Q(2)|$ of sections of $|\O_Q(2)|$ containing both $g(Z_1)$ and $g(Z_2)$. Further, let $\H_1, \H_2 \subset |\O_Q(1)|$ be the linear subsystems of sections of $|\O_Q(1)|$  containing $g(Z_1)$ and $g(Z_2)$ respectively. Then $U_1 +U_2 \in \mathcal{Q}$, where $U_i\in \H_i$, and $\mathrm{Bs}|U_1+U_2|=\{g(Z_1), g(Z_2)\}$. Since one has $|U_1+U_2| \subset \mathcal{Q}$ and $g^*(\mathcal{Q})\subset \L$, then $\mathrm{Bs}\L\subseteq\{Z_1,\nu(Z_1), Z_2,\nu(Z_2)\}$.

		Therefore, in order to prove that $\beta^*(-2K_V)-E_1-E_2$ is nef, it is enough to show that it intersects positively with the stricts transforms of the curves $\nu(Z_i)$ via $\beta$ and curves contained in the $\beta$-exceptional divisors. 
		By construction,
		\[
		-2K_V\cdot \nu (Z_i)=2.
		\]
		Consequently,
		\[
		(\beta^*(-2K_V)-E_1-E_2)\cdot \beta_*^{-1}(\nu (Z_i))\geq2 -1-1\geq 0.
		\]
		
		Now let's consider curves contained in the $\beta$-exceptional divisors. Let $\N_{Z_i/V}$ be the normal sheaf of the curve $Z_i$ on $V$. Hence,
		\[
			\N_{Z_i/V}\cong \O_{\P^1}(a)\oplus \O_{\P^1}(b)
		\]
		for some integers $a$, $b$ and $a\geq b$. Using the exact sequence 
		\[
		\begin{CD}
			0@>>>\mathcal{T}_{Z_i} @>>>\mathcal{T}_{V}|_{Z_i} @>>> \N_{Z_i/V}@>>>0,
		\end{CD}
		\]
		one considers the top exterior powers to get
		\[
		\begin{CD}
			0@>>>\O_{Z_i}(-K_{Z_i}) @>>> \O_{V}(-K_V)|_{Z_i} @>>> \det\N_{Z_i/V}@>>>0.
		\end{CD}
		\]
		Since 
		\[
		\det \N_{Z_i/V}\cong \det (\O_{\P^1}(a)\oplus \O_{\P^1}(b))\cong \O_{\P^1}(a)\otimes \O_{\P^1}(b)\cong \O_{\P^1}(a+b),
		\]
		then by taking degrees of the previous exact sequence one has
		\[
		a+b=-K_V\cdot Z_i +2g(Z_i)-2=-1.
		\]
		
		Let $S$ be a general hyperplane section of $V$ containing the curve $Z_i$. Then $S$ is a smooth $K3$ surface, because $Z_i$ is not contained in the ramification divisor of $g$. Therefore, $\N_{Z_i/S}\cong\O_{\P^1}(-2)$. Using the exact sequence 
		\[
		\begin{CD}
			0@>>>\N_{Z_i/S} @>>>\N_{Z_i/V} @>>> \N_{S/V}|_{Z_i}@>>>0,
		\end{CD}
		\]
		one has $b\geq -2$. In particular, $a-b\leq 3$. Let $\s_i$ be the exceptional section of the Hirzebruch surface $\beta|_{E_i}\colon E_i \to Z_i$. Then
		\[
		(\beta^*(-2K_V)-E_1-E_2)\cdot \s_i = \frac{5+b-a}{2}> 0.
		\]
		Thus, $\beta^*(-2K_V)-E_1-E_2$ is nef.

		Now let $\M_W = \beta_*^{-1}(\pi^* \M_X)$. Denote as $m_i=\mult_{Z_i}(\pi^*\M_X)$. Then, on the one hand
		\begin{align*}
		&(\beta^*(-\tfrac{1}{\lambda}K_V)-m_1E_1-m_2E_2)^2\cdot(\beta^*(-2K_V)-E_1-E_2)\\
		&=-3m_1^2-3m_2^2-\frac{2m_1}{\lambda}-\frac{2m_2}{\lambda}+\frac{8}{\lambda^2}<0,
		\end{align*}
		but on the other hand, for general surfaces $S_1$ and $S_2$ from $\M_W$,
		\begin{align*}
		&(\beta^*(-\tfrac{1}{\lambda}K_V)-m_1E_1-m_2E_2)^2\cdot(\beta^*(-2K_V)-E_1-E_2)\\
		&=S_1\cdot S_2 \cdot (\beta^*(-2K_V)-E_1-E_2)\geq 0.
		\end{align*}
		Contradiction.
		\end{proof}

\begin{lemma}\label{lem2}
	Assume that $a_i\geq 0$ for some $i$. Then $\Phi_i$ is a square equivalence. 
\end{lemma}
\begin{proof}
Let $\kappa_i\colon \widetilde{X}_{i}\to X$ be the Kawamata blow-up of $q_i$ and $E_i$ be the $\kappa_i$-exceptional divisor. Then
	\[
(-K_{U_i})^3=(-K_{\widetilde{X}_{i}})^3=(\kappa_i^*(-K_X)-\tfrac{1}{2}E_i)^3=\kappa_i^*(-K_X)^3-\tfrac{1}{8}E_i^3=3/2,
	\]
and 
	\[
 		\nef (U_i/\P^1)=0
	\]
by Remark \ref{rem}, because in our case $\xi_i$ is the class of the flopping curves, so $-K_{U_i}\cdot \xi_i=0$. Thereby, $U_i/\P^1$ satisfies the $K_{3/2}^3$-condition, and we are done by Proposition \ref{okada}.
\end{proof}

\begin{lemma}\label{lem3}
	Assume that all $a_i<0$. Then $\Phi\colon X \dashrightarrow W$ is an isomorphism.
\end{lemma}
\begin{proof}
	Assume the contrary. Then we have the following commutative diagram:
\[
	\xymatrix{
		&W&\\
	X\ar@{-->}[ru]^{\Phi}\ar@{-->}[rr]^{\rho_i}&& U_i \ar@{-->}[lu]_{\Phi_i}\\
	&\widetilde{X}_i\ar[lu]^{\kappa_i}\ar@{-->}[ru]_{\psi_i}&
	}
\]	
Notice $\rho_i \colon X\dashrightarrow U_i$ is the composition of the inverse map to the Kawamata blow-up $\kappa_i \colon \widetilde{X}_i \to X$ and two flops $\psi_i \colon \widetilde{X}_i\dashrightarrow U_i$. Consequently, without loss of generality, we can identify linear systems on $U_i$ with linear systems on $\widetilde{X}_i$ via $(\psi_i)_*$. We write $\M_X\sim_{\Q}d H$, and hence we have 
\begin{align*}
&\kappa_i^*H=H_{\widetilde{X}_i}+E_i;\\
&K_{\widetilde{X}_i}=-H_{\widetilde{X}_i}-\tfrac{1}{2}E_i;\\
&\M_{\widetilde{X}_i}=\kappa_i^*\M_{X}-m_iE=dH_{\widetilde{X}_i}+(d-m_i)E_i.
\end{align*}
Therefore,
\[
dH_{\widetilde{X}_i}+(d-m_i)E_i=-n_iK_{\widetilde{X}_i}+a_iF=(n_i+a_i)H_{\widetilde{X}_i}+\tfrac{n_i}{2}E_i.
\]
Thus, 
\begin{align}\label{canon}
	d-2m_i=-a_i.
\end{align}
By \cite[Lemma 7]{kawamata1996divisorial}, Lemmas \ref{lem1}, \ref{lemma:curves}, a point $q_j$ must be a centre of non-canonical singularities of the pair $(X,\tfrac{1}{d}\M_X)$ for some $j$. However all $-a_i>0$, so the pair $(X,\tfrac{1}{d}\M_X)$ is canonical, and thus $\Phi$ is an isomorphism by the Noether-Fano inequality.
\end{proof}

\begin{lemma} \label{lemma:singular_points}
	Let $q_i$ be a centre of non-canonical singularities of the pair $(X, \lambda \M_X)$. Then $\rho_i\colon U_i\dashrightarrow W$ a square equivalence.  
\end{lemma}	 
\begin{proof}
Follows directly from $\eqref{canon}$ and Proposition \ref{okada}.
\end{proof}

Finally, Lemmas \ref{lem1}, \ref{lemma:curves}, \ref{lem2}, \ref{lem3}, \ref{lemma:singular_points}  and the Noether-Fano inequality imply the following theorem.

\begin{theorem}\label{theorem:final}
	Let $\Phi\colon X \dashrightarrow W$ be a birational map to a Mori fibre space $W/B$. Then one of the following holds:
	\begin{itemize}
		\item[\textup{1)}] $B=\mathrm{pt}$ and $\Phi$ is an isomorphism;
		\item[\textup{2)}] $B\cong \P^1$ and for $i=1,\ldots,8$ there exists the following commutative diagram:
	\[
	\xymatrix{X\ar@{-->}_{\rho_i}[dr] \ar@{-->}@/^1pc/[rrd]^{\Phi}&\\
		&U_i\ar@{-->}[r]^{\chi} \ar[d]_{f_i}&W\ar[d]\\
		&\P^1\ar^{\omega}[r]&B
	}
    \]	
    where $\chi\colon U_i \dashrightarrow W$ is a birational map and $\omega \colon \P^1\to B$ is an isomorphism.
	\end{itemize}
\end{theorem}

Now, Main Theorem \ref{main2} implies from the construction provided in Section \ref{dp1} and Theorem \ref{theorem:final}.

\begin{corollary}
	In the notation and assumptions above the following holds: 
	\begin{itemize}
		\item[\textup{1)}] $2\leq|\mathcal{P}(X)|\leq 9$. For a general $X$ one has $|\mathcal{P}(X)|=9$;
		\item[\textup{2)}] $\Aut(X)=\mathrm{Bir}(X)$;
		\item[\textup{3)}] $X$ is not birational to a conic bundle;
		\item[\textup{4)}] $X$ is not rational.
	\end{itemize}
\end{corollary}
\begin{proof}
	Assertions 2), 3), 4) are obvious. Let us proof 1). 
	
	Clearly $2\leq|\mathcal{P}(X)|\leq 9$, so we need to proof that $|\mathcal{P}(X)|=9$ for a general $X$. Assume that there is a birational map $g\colon U_1\dashrightarrow U_2$. Then $g$ induces a birational automorphism $h\colon X \dashrightarrow X$, where $h=\rho_1\circ g\circ\rho_2^{-1}$. However, by the assertion 2), $h$ is biregular. Consequently, $h$ must permute the points $q_1$ and $q_2$. Now let's consider the map determined by the linear system $|H|$:
	\[
		\phi_{|H|}\colon X \dashrightarrow \P^2=\P(H^0(X,\O_X(H))).
	\]
	
	Hence, $h$ induces a biregular action on $\P^2=\P(H^0(X,\O_X(H)))$. Now we notice that for a general $X$ the images of $q_i$ on $\P^2$ cannot be non-trivially permuted by any element of $\mathrm{PGL}(3,\mathbb{C})$. Hence, $|\mathcal{P}(X)|=9$ for a general $X$. Example \ref{exam} finishes the proof.
\end{proof}

\begin{example}\label{exam}
Let $x_0,\ldots$, $x_4, x_5$ be homogeneous coordinates on $\P(1^5,2)$, where $x_5$ is the coordinate of weight $2$. Let $V\subset \P(1^5,2)$ be defined by the following equations:
	\[
		\begin{cases}
			&x_0x_1-x_2^2+x_3x_4=0,\\
			&x_5^2+x_0x_1x_3^2+x_2^2x_4^2+
			2x_0^3x_1 - 4 x_0 x_1^3 + 30 x_1^4 - 42 x_0^3 x_2 + 5 x_0^2 x_1 x_2 - 2 x_0 x_1^2 x_2 + 79 x_1^3 x_2 \\
 			&+363 x_0^2 x_2^2 + x_0 x_1 x_2^2 - 803 x_1^2 x_2^2 - 890 x_0 x_2^3 + 850 x_1 x_2^3 + 411 x_2^4+x_3^4+x_4^4=0.
		\end{cases}
	\]
		Let involution $\sigma\colon V\to V$ be given by
	\[
		\sigma (x_0:x_1:x_2:x_3:x_4:x_5)=(x_0:x_1:x_2:-x_3:-x_4:-x_5).
	\]	
	Then one can see that the $\sigma$-invariant points are the following:
	\begin{align*}
		&p_1=(1: 0: 0:0:0:0);\\
		&p_2=(1: 1:1:0:0:0);\\
		&p_3=(1: 1:-1:0:0:0);\\
		&p_4=(4: 1:2:0:0:0);\\
		&p_5=(1: 9:3:0:0:0);\\
		&p_6=(9: 4:6:0:0:0);\\
		&p_7=(25: 1:5:0:0:0);\\
		&p_8=(1: 49:-7:0:0:0).
	\end{align*}
	Let $r_i=(\phi_{|H|}\circ\pi)(p_i)$. Then $\Sigma =\{r_1,\ldots, r_8\}= Y_1\cap Y_2$, where 
	\begin{align*}
	&Y_1=\{x_0x_1-x_2^2=0\};\\
	&Y_2=\{2x_0^3x_1 - 4 x_0 x_1^3 + 30 x_1^4 - 42 x_0^3 x_2 + 5 x_0^2 x_1 x_2 - 2 x_0 x_1^2 x_2 + 79 x_1^3 x_2\\
	&+363 x_0^2 x_2^2 + x_0 x_1 x_2^2 - 803 x_1^2 x_2^2 - 890 x_0 x_2^3 + 850 x_1 x_2^3 + 411 x_2^4=0\}.
	\end{align*}
	
	Let $M\colon \P^2 \to \P^2$ be an automorphism such that $M(\Sigma)=\Sigma$. Obviously  $M(Y_1)=Y_1$, so one can consider $r_1,\ldots, r_8$ as points on $\P^1$. Now direct computations show that there are no automorphisms of $\P^1$ respecting this set of points.
		\end{example}
		
In fact, for a special $X$ we may have $|\mathcal{P}(X)|=2$. Consider the following example.

\begin{example}\label{exam1}
	Let $V$ be defined by the following equations:
	\[
		\begin{cases}
			-2ix_0x_1+x_2^2+x_3x_4=0,\\
			x_5^2+x_0x_1x_3^2+x_2^2x_4^2+x_0^4+x_1^4+x_2^4+x_3^4+x_4^4=0,
		\end{cases}
	\]
	and the involution $\sigma$ is the same.

	Then one can see that the $\sigma$-invariant points are the following:
	\begin{align*}
			& p_1=((-1 - i)(2 + \sqrt{3})^{1/4}: (1 + i) (2 - \sqrt{3})^{1/4}: 2:0:0:0);\\
		&p_2=((1 - i)(2 + \sqrt{3})^{1/4}: (1 - i)(2 - \sqrt{3})^{1/4}:2:0:0:0);\\
		&p_3=((1 + i) (2 + \sqrt{3})^{1/4}: (-1 - i) (2 - \sqrt{3})^{1/4}:2:0:0:0);\\
		&p_4=((-1 + i) (2 +\sqrt{3})^{1/4}, (-1 + i) (2 - \sqrt{3})^{1/4}, 2:0:0:0);\\
		&p_5=((-1 - i) (2 - \sqrt{3})^{1/4}: (1 + i) (2 + \sqrt{3})^{1/4}:2:0:0:0);\\
		&p_6=((1 - i) (2 - \sqrt{3})^{1/4}: (1 - i) (2 + \sqrt{3})^{1/4}:2:0:0:0);\\
		&p_7=((-1 + i) (2 - \sqrt{3})^{1/4}: (-1 + i) (2 + \sqrt{3})^{1/4}:2:0:0:0);\\
		&p_8=((1 + i) (2 - \sqrt{3})^{1/4}: (-1 - i) (2 + \sqrt{3})^{1/4}:2:0:0:0).
	\end{align*}
	Consider the following morphisms:
	\[
		\alpha\colon \P(1^5,2)\to\P(1^5,2)
	\]
	given by
	\[
		\alpha(x_0:x_1:x_2:x_3:x_4:x_5)=(ix_0:-ix_1:x_2:x_3:x_4:x_5),
	\]
	and 
	\[
		\beta\colon \P(1^5,2)\to\P(1^5,2)
	\]
	given by
	\[
		\beta(x_0:x_1:x_2:x_3:x_4:x_5)= (x_1:x_0:x_2:x_3:x_4:x_5).
	\]
	Then $\alpha$ and $\beta$ induce automorphisms of $V$, which we will still denote as $\alpha$ and $\beta$. Denote as $\Omega=\{p_1,\ldots, p_8\}$. Representing the point $p_i$ by number $i$, we notice that $\alpha$ acts on the set of eight elements $\Omega$ as $(1234)(5678)\in \mathfrak{S}_8$, and $\beta$ acts on $\Omega$ as $(18)(26)(35)(47)\in \mathfrak{S}_8$. Therefore, the subgroup $\langle \alpha, \beta\rangle \subset \Aut(V)$ permutes transitively the set $\Omega$. Thus, all $U_i/\P^1$ are square equivalent and hence $|\mathcal{P}(X)|=2$. 
\end{example}

\bibliography{References}
\bibliographystyle{alpha}	

\end{document}